\documentclass[12pt]{amsart}

\usepackage{amsmath,amstext,amsgen,amsbsy,amsopn}
\usepackage{amssymb,amsfonts,amsthm}
\usepackage{latexsym,amsxtra,euscript,amscd}
\usepackage{array}
\usepackage{mathtools}
\usepackage{xcolor}
\usepackage{graphics,graphicx}
\usepackage{enumitem}

\newtheorem{theorem}{Theorem}
\newtheorem{lemma}{Lemma}
\newtheorem{corollary}{Corollary}
\newtheorem{prop}{Proposition}

\begin{document}

\centerline{{\bf Representations with $\mathbf{k}$-generalized Fibonacci numbers}}
\medskip

\centerline{Taboka Chalebgwa\footnote{ORCID: 0000-0002-6132-6368; Department of Mathematics and Applied Mathematics, University of Pretoria, Hatfield, Pretoria, 0028, South Africa} -- L\'aszl\'o Szalay\footnote{ORCID: 0000-0002-4582-6100; Department of Mathematics, J.~Selye University, SK-94501 Kom\'arno, Bratislavsk\'a cesta 3322, Slovakia; and Department of Mathematics, Institute of Basic Sciences, University of Sopron, 9400 Sopron, Bajcsy-Zsililnszky utca 4, Hungary.} }
\medskip

\noindent{\it Abstract.} 
We study representations of integers using $k$-generalized Fibonacci numbers. For $k\geq 2$, we first consider signed representations of zero with coefficients in $\{-1,0,1\}$ and give a recursive description of their number. The resulting counting sequences satisfy linear recurrences whose characteristic polynomials are determined explicitly. In the
Fibonacci and Tribonacci cases, these recurrences reveal unexpected connections between the corresponding representation counts and Tribonacci and Fibonacci sequences, respectively.

We then consider representations with coefficients in $\{0,1\}$ in the Tribonacci case. Using a random inhomogeneous Tribonacci recurrence, we construct a binary-tree model in which representation multiplicities are encoded by a family of polynomials satisfying the product formula $Q_n(x)=\prod_{k=2}^{n-1}(1+x^{T_k})$. This product admits a probabilistic interpretation in terms of weighted Bernoulli sums. After normalization by $T_n$, these sums converge in
distribution to the Bernoulli convolution $\sum_{j=1}^{\infty}\varepsilon_j\rho^{-j}$, where $\rho$ is the Tribonacci constant and the $\varepsilon_j$ are independent Bernoulli random variables. The limiting distribution satisfies a natural self-similarity relation, linking the representation problem to self-similar measures associated with Tribonacci scaling.

 \section{Introduction}

We examine the number of representations of integers by the so-called $k$-generalized Fibonacci numbers, sometimes referred to as the ``$k$-bonacci" numbers. For such a sequence, which we shall denote by $(a_n)$ locally, two distinct but related problems are investigated:
\begin{itemize}[leftmargin=0.2in]
	\item 
	Firstly, fixing a positive integer $n$, we are interested in the number of solutions to the diophantine equation 
	$$\alpha_0 a_0+\alpha_1 a_1+\cdots+\alpha_n a_n=0,$$
	with the coefficients $\alpha_i\in\{-1,0,1\}$.
	We solve this problem completely by giving a recursive description for the number of solutions.
	\item In the second problem, we restrict ourselves to the case of ``unsigned coefficients" $\alpha_i\in\{0,1\}$. For this variation on the theme, given a positive integer $N$, we investigate the number of linear combinations satisfying
	$$\alpha_0 a_0+\alpha_1 a_1+\cdots+\alpha_n a_n = N.$$
	
\end{itemize}

\vspace{0.2cm}

Our main results show that the number of signed representations satisfies a linear recurrence whose characteristic polynomial can be determined explicitly. In the unsigned Tribonacci case, the representation frequencies are encoded by a family of polynomials admitting the product formula $Q_n(x)=\prod_{k=2}^{n-1}(1+x^{T_k})$, where $T_k$ denotes the $k$th Tribonacci number.
\vspace{0.2cm}

Indeed, while representations based on the Fibonacci sequence have been widely studied, considerably less is known when it comes to higher-order generalizations. In particular, enumerative results for signed and restricted representations involving $k$-generalized Fibonacci numbers are somewhat scarce in the literature. One approach to the second problem was explored in \cite{LSz}, for the special case of the Fibonacci sequence. The sought after number of such representations appeared as the coefficients of a recursively defined polynomial sequence.
\vspace{0.2cm}

Inspired by the circle of ideas from \cite{LSz}, we handle the case of the Tribonacci numbers. In principle, our method could work for any $k\ge2$, however, the computations involved get more technical as $k$ increases. It is for this reason (and as a ``proof of concept,") that we restrict ourselves only to the analogous case of $k=3$, and solve it completely.
\vspace{0.2cm}

On the other hand, we also highlight that the arguments used throughout the paper rely on dominance properties of $k$-generalized Fibonacci numbers similar in spirit to those appearing in Zeckendorf-type decomposition results (The Zeckendorf representation was first investigated by Lekkerkerker \cite{Le}). We deviated from this classical setting on two fronts: on the one hand, we allow signed coefficients and non-unique representations, and on the other hand, we pursue enumerative as opposed to uniqueness results.
\vspace{0.2cm}

With the above preamble in mind, we now proceed to remind the reader of some of the requisite technical definitions, and set the stage for our results.
\vspace{0.2cm}

For a positive integer $k\ge2$, the $k$-generalized Fibonacci sequence $(F_{n}^{(k)})_{n\in\mathbb{N}}$ has initial values
\begin{equation*}\label{inivalues}
	F_{0}^{(k)}=F_{1}^{(k)}=\cdots=F_{k-2}^{(k)}=0,\; F_{k-1}^{(k)}=1,
\end{equation*}
and satisfies the recurrence
\begin{equation}\label{reck}
	F_{n}^{(k)}=F_{n-1}^{(k)}+F_{n-2}^{(k)}+\cdots+F_{n-k}^{(k)}\qquad {\rm for}\,{\rm all}\;n\ge k.
\end{equation}

We remark that, in some literature, the sequence is sometimes represented alternatively via a translation of the subscripts by $k-2$ in initial values. That is: $F_{0}^{(k)}=0$, $F_{1}^{(k)}=1$, and so forth.
\vspace{0.2cm}

Using (\ref{reck}), the following ``observation" can be immediately deduced:
\begin{equation*}\label{nextvalues}
	F_{k}^{(k)}=1,\;F_{k+1}^{(k)}=2,\;F_{k+2}^{(k)}=4,\;\cdots,\; F_{2k-2}^{(k)}=2^{k-2},\; F_{2k-1}^{(k)}=2^{k-1}.
\end{equation*}

When $k=2$, one recovers the Fibonacci sequence $(F_n^{(2)})=(F_n)$, while the case $k=3$ is referred to as the Tribonacci sequence, denoted by $(F_n^{(3)})=(T_n)$. Its recursion is defined by $T_0=T_1=0$, $T_2=1$, and $T_n=T_{n-1}+T_{n-2}+T_{n-3}$ for $n \ge 3$. 
\vspace{0.2cm}

Proceeding, since our second (partition) problem eventually boils down to examining the behavior of an inhomogeneous linear recurrence whose inhomogeneous part is a random value from the Boolean set $\{a,b\}$, in order to discuss the solution, we need the notion of a random recurrence sequence. To this end:

Let $a<b$ be two arbitrary real numbers. Suppose that the terms $w_{n}=w_n(a,b)$ of a random sequence $(w_n)_{n \ge 0}$ are generated by, say, a coin toss, whereby $w_n = a$ with probability $p\in[0,1]$ (if the outcome is heads, for instance), and $w_n = b$ with probability $q=1-p$ (in case of tails).\\

We study the random inhomogeneous Tribonacci recurrence
\begin{equation}\label{rule}
	G_{n}=G_{n-1}+G_{n-2}+G_{n-3}+w_{n-3}(a,b)\qquad(n\ge3)
\end{equation}
with initial values $G_0=G_1=0$, $G_2=1$.
\vspace{0.2cm}

Clearly, the extreme cases where $p=1$ or $p = 0$ produce the deterministic sequences $G_{n}=G_{n-1}+G_{n-2}+G_{n-3}+a$ and	$G_{n}=G_{n-1}+G_{n-2}+G_{n-3}+b$, respectively.
\vspace{0.2cm}

In accordance with (\ref{rule}), $(G_n)_{n \geq 2}$ may take two possible values: either $G_{n}=G_{n-1}+G_{n-2}+G_{n-3}+a$ or $G_{n}=G_{n-1}+G_{n-2}+G_{n-3}+b$. Naturally, this particular scenario can be captured via a binary tree denoted by $({\mathcal T}_{n,k})$. The random inhomogeneous recurrence provides a convenient way to generate all admissible $0-1$ combinations simultaneously. Each path in the associated binary tree corresponds to a choice of coefficients, and repeated values reflect multiple representations. Figure \ref{fig1} illustrates the first few levels of the tree. 
\begin{figure}[h]
	\centering
	\includegraphics[scale=0.33]{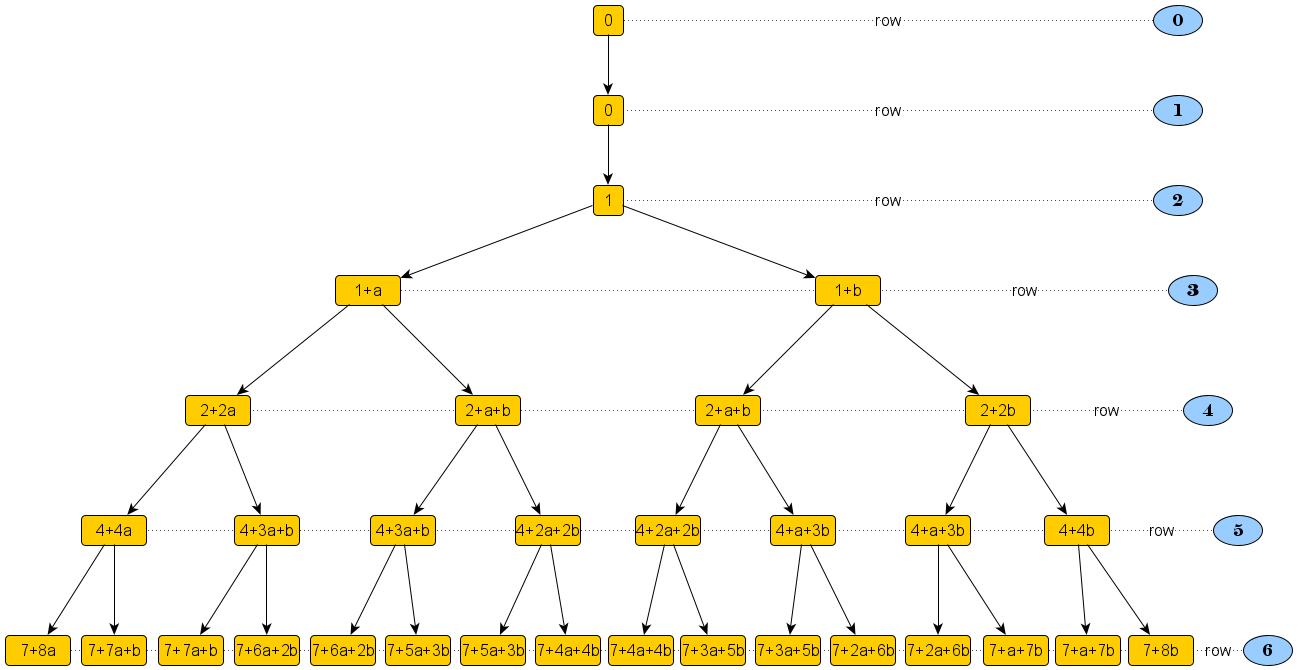} 
	\caption{Binary tree linked to (\ref{rule}) with parameters $a$ and $b$.}
	\label{fig1}
\end{figure}

As it turns out, the values (and frequencies thereof,) appearing in each row or level of the above tree can be described rather precisely. In particular, as will be shown in the paper, these frequencies are coefficients of certain well defined polynomials.
\vspace{0.2cm}

The paper is organized as follows: In Section 2, we collect all the preliminary properties of the $k$-generalized Fibonacci sequences. We then establish a summation identity required in subsequent arguments. Section 3 is devoted to the resolution of the signed representations problem. We achieve this by first deriving a linear recurrence satisfied by the number of admissible representations, then examining its behavior for the lower order cases. Section 4 deals with the $\{0,1\}$ representation problem in the Tribonacci case, where a random inhomogeneous recurrence leads to a binary-tree model whose value frequencies are encoded by explicitly determined polynomials. Finally, in Section 5 we give these polynomials a probabilistic interpretation in terms of weighted Bernoulli sums. After normalization by the $T_n$, we prove convergence in
distribution to a Bernoulli convolution and derive the corresponding self-similarity equation for the limiting distribution.
\vspace{0.2cm}

We now proceed to Section 2, where we compile all the necessary technical lemmas.
 
\section{Preliminaries}

\begin{lemma}\label{l1}
	If $n\ge k-1$, then 
	$$
	2F_{n}^{(k)}\ge F_{n+1}^{(k)}.
	$$
\end{lemma}
\begin{proof}
When $n=k-1$ we see $2F_{k-1}^{(k)}=2\cdot1\ge1=F_{k}^{(k)}$. 
For $n\ge k$ consider
the difference $F_{n+1}^{(k)}-F_{n}^{(k)}$. This satisfies
$$
F_{n+1}^{(k)}-F_{n}^{(k)}=\sum_{j=n+1-k}^{n-1}F_{j}^{(k)}\le\sum_{j=n-k}^{n-1}F_{j}^{(k)}=F_{n}^{(k)},
$$
which implies the statement of the lemma.
\end{proof}

\noindent Remark: Note that, the statement is in fact true for all $n\ge0$ except for the case $n=k-2$.

\begin{lemma}\label{l2}
For $n\ge k-1$ the inequality
$$
\sum_{j=0}^{n}F_{j}^{(k)}<F_{n+2}^{(k)}
$$
is true.
\end{lemma}

\begin{proof}
We proceed by induction. If $n=k-1$, we see that $F_{k-1}^{(k)}=1<2=F_{k+1}^{(k)}$.	
Assume now that the statement is true for a given value of $n\ge k-1$. Then
$$
\sum_{j=0}^{n+1}F_{j}^{(k)}=\sum_{j=0}^{n}F_{j}^{(k)}+F_{n+1}^{(k)}<F_{n+2}^{(k)}+F_{n+1}^{(k)},
$$	
which does not exceed $F_{n+3}^{(k)}$. Indeed, the inequality 
$F_{n+1}^{(k)}+F_{n+2}^{(k)}\le F_{n+3}^{(k)}$
follows from the non-negativity of the sequence $(F_{n}^{(k)})$, in combination with the defining recurrence rule (\ref{reck}).
\end{proof}

Until this point, all the priorly established results provided inequalities for sums of the $k$-generalized Fibonacci numbers. The next statement provides an explicit and closed formula for the first terms of the $k$-generalized Fibonacci sequence.

\begin{theorem}\label{sumk}
	\begin{equation}\label{ksum}
	\sum_{j=0}^{n}F_{j}^{(k)}=\frac{F_{n+k-1}^{(k)}+F_{n}^{(k)}-1-\sum_{i=1}^{k-3}iF_{n+k-2-i}^{(k)}}{k-1}.
	\end{equation}
\end{theorem}
\begin{proof}

	As usual, we proceed by induction on $n$. Firstly, if $n=0$, then 
	$$
	F_{0}^{(k)}=\frac{F_{k-1}^{(k)}+F_{0}^{(k)}-1-\sum_{i=1}^{k-3}iF_{k-2-i}^{(k)}}{k-1}=\frac{1-0-1-0}{k-1}=0.
	$$
	Assume now that (\ref{ksum}) holds for some $n\ge0$. Then for $n+1$ we have
\begin{eqnarray*}
	\sum_{j=0}^{n+1}F_{j}^{(k)}&=&F_{n+1}^{(k)}+\frac{F_{n+k-1}^{(k)}+F_{n}^{(k)}-1-\sum_{i=1}^{k-3}iF_{n+k-2-i}^{(k)}}{k-1}\\
	&=&\frac{(k-1)F_{n+1}^{(k)}+F_{n+k-1}^{(k)}+F_{n}^{(k)}-1-\sum_{i=1}^{k-4}iF_{n+k-2-i}^{(k)}-(k-3)F_{n+1}^{(k)}}{k-1}\\
	&=&\frac{2F_{n+1}^{(k)}+(F_{n+k}^{(k)}-F_{n+k-2}^{(k)}-\cdots-F_{n+1}^{(k)}-F_{n}^{(k)})+F_{n}^{(k)}-1-\sum_{i=1}^{k-4}iF_{n+k-2-i}^{(k)}}{k-1}\\
	&=&\frac{F_{(n+1)+k-1}^{(k)}+F_{n+1}^{(k)}-1-\sum_{i=1}^{k-3}iF_{(n+1)+k-2-i}^{(k)}}{k-1}.	
\end{eqnarray*}
In the above computation, we used only the recurrence rule of the $k$-generalized Fibonacci sequence line (3), and straightforward algebraic simplifications.
\end{proof}

\begin{corollary}\label{C1}
For small values of $k$ we derive the following formulae.	
\begin{itemize}
	\item If $k=2$, then the well-know Fibonacci identity (with $F_n=F_n^{(2)}$) follows (see page 203, \cite{L}):
	$$
	\sum_{j=0}^{n}F_{j}=F_{n+1}+F_n-1=F_{n+2}-1.
	$$
	\item For $k=3$ in Theorem \ref{sumk}, with the notation $T_n=F_n^{(3)}$ we recover the well-known formula 
	$$
	\sum_{j=0}^nT_j=\frac{T_{n+2}+T_n-1}{2},
	$$
	see \cite{K} (where the author used a shifted version of Tribonacci sequence).\\
	
	\item For $k=4$ we derive
	$$
	\sum_{j=0}^nF_{j}^{(4)}=\frac{F_{n+3}^{(4)}+F_{n}^{(4)}-1-F_{n+1}^{(4)}}{3}.
	$$
\end{itemize}
\end{corollary}

In the discussion that follows, we recall a ``structural identity" for linear recurrences which will allow us to relate a homogeneous recurrence to its inhomogeneous counterpart. This identity may be viewed as a discrete analogue of the ``variation of parameters" method. More specifically, we express the solution of an inhomogeneous recurrence in terms of the associated homogeneous sequence and forcing terms. The result enables us to obtain explicit expressions for the random inhomogeneous Tribonacci sequence we will be studying later.\\

Given a positive integer $\ell$ and complex numbers $f_0,f_1,\dots,f_{\ell-1}$, define the linear homogeneous recurrence  
\begin{equation}\label{f}
	f_n=A_1f_{n-1}+A_2f_{n-2}+\cdots+A_\ell f_{n-\ell}\qquad (n\ge\ell),
\end{equation}
where the coefficients $A_1,\dots,A_{\ell-1},A_\ell\ne0$ are fixed complex numbers. Moreover suppose that $(w_n)\in\mathbb{C}^\infty$ is an arbitrary sequence. Consider now the ``associated" linear recurrence sequence $(G_n)$ constructed via
\begin{equation}\label{G}
	G_n=A_1G_{n-1}+A_2G_{n-2}+\cdots+A_\ell G_{n-\ell}+w_{n-\ell}\qquad(n\ge\ell),
\end{equation}
where the complex initial values $G_0,\dots,G_{\ell-1}$ are assumed to be given. Note that formulae (\ref{f}) and (\ref{G}) essentially differ only in the terms of $(w_n)$.\\

Theorem 1 of \cite{BRSz} has the following statement.

\begin{theorem}\label{Sopron}
For $n\ge \ell$ the terms of the sequences $\{f_n\}$, $\{w_n\}$, and $\{G_n\}$ satisfy the identity
\begin{equation*}\label{geneq}
	\sum_{j=0}^{\ell-1}f_jG_{n+\ell-j}=\sum_{j=0}^{\ell-1}\sum_{i=0}^{\ell-1-j}f_{n-j}A_{j+1+i}G_{\ell-1-i}+
	\sum_{j=0}^{\ell-2}\sum_{i=1}^{\ell-1-j}f_{j}A_{i}G_{n+\ell-j-i}+\sum_{j=0}^{n}f_{n-j}w_j.
\end{equation*}
\end{theorem}

For a very specific set of initial conditions, the following corollary is deduced.

\begin{corollary}\label{C2}
Let $f_0=f_1=\dots=f_{\ell-2}=0$, $f_{\ell-1}=1$, and similarly, $G_0=G_1=\dots=G_{\ell-2}=0$, $G_{\ell-1}=1$, moreover let $A_1=A_2=\cdots=A_\ell=1$. Assuming a constant sequence $w_n=w$, we have that for all $n\ge\ell$

	\begin{equation*}\label{eC2}
		G_{n}=\sum_{j=n-\ell}^{n-1}f_j+w\sum_{j=0}^{n-1}f_{j}.
	\end{equation*}
\end{corollary}

Additionally, by letting $\ell  = 3$ in the above Corollary \ref{C2}, in other words, ~$f_n=T_n$, and then applying Corollary \ref{C1}, we get the following corollary.

\begin{corollary}
\begin{equation*}\label{eC2}
	G_{n}=T_{n}+w\frac{T_{n+1}+T_{n-1}-1}{2}.
\end{equation*}	
\end{corollary}

Proceeding, for $w=a$ or $w=b$, we use the above formula to then define the quantities:
$$
m_n(a)=T_{n}+a\frac{T_{n+1}+T_{n-1}-1}{2}\quad{\text{and}}\quad m_n(b)=T_{n}+b\frac{T_{n+1}+T_{n-1}-1}{2}.
$$

Having laid down all of the foundational preamble, we can now finally move on to the proofs of our main theorems.

\section{The k-generalized Fibonacci sequence and certain representations}

In this section, we address the first partition problem from the Introduction.

We begin with the formal definition of the set of $k$-Fibonacci representations of the number zero with the possible coefficients having value from $\{-1,0,1\}$. Towards this, let $n\ge k-1$ and 
\begin{equation*}
	R^{(k)\star}_n=\{(\varepsilon_{k-1},\varepsilon_{k},\dots,\varepsilon_n)\;|\;\varepsilon_{k-1}F_{k-1}^{(k)}+\varepsilon_{k}F_{k}^{(k)}+\cdots+\varepsilon_nF_n^{(k)}=0, \varepsilon_i\in\{-1,0,1\}\}.
\end{equation*}
We further define $R_n^{(k)}\subset (R^{(k)\star}_n\setminus(0,0,\dots,0\
))$ with the property that: if $z$ is the largest subscript in $(\varepsilon_{k-1},\varepsilon_{k},\dots,\varepsilon_n)$ such that $\varepsilon_z$ is non-zero, then $\varepsilon_z=1$. Finally, put $r^\star_n=|R^{(k)\star}_n|$ and $r_n=|R^{(k)}_n|$, where, for brevity, if there is no risk of confusion, we omit the parameter $k$ in the cardinalities. \\

In the previous discussion, we introduced the set $R^{(k)}_n$ specifically in order to address the ``symmetrically" trivial cases: If $(\varepsilon_{k-1},\varepsilon_{k},\dots,\varepsilon_n)=(0,0,\dots,0)$, then we get a trivial solution to $\varepsilon_{k-1}F_{k-1}^{(k)}+\varepsilon_{k}F_{k}^{(k)}+\cdots+\varepsilon_nF_n^{(k)}=0$, whereas if $\sum_{j=k-1}^n\varepsilon_jF^{(k)}_j=0$, then $\sum_{j=k-1}^n(-\varepsilon_j)F^{(k)}_j=0$ as well. It is thus apparent that $r^\star_n=2r_n+1$.\\

The following lemma expresses a dominance property of 
$k$-generalized Fibonacci numbers analogous to the inequalities used in Zeckendorf-type uniqueness proofs. Ro\-ughly speaking, once a sufficiently large term appears (in the representation) with positive sign, it forces a rigid configuration among the preceding coefficients. While interesting in its own sake, the lemma will be crucial in our proof of Theorem \ref{final}.

\begin{lemma}\label{crucial}
	Assume that $(\varepsilon_{k-1},\varepsilon_{k},\dots,\varepsilon_n)$ is an element of $R_n^{(k)}$. If $\varepsilon_z=1$ holds for some $2k-1\le z\le n$ such that $\sum_{j=z+1}^n\varepsilon_jF^{(k)}_j=0$, then
	\begin{enumerate}
		\item $\varepsilon_{z-1}=\varepsilon_{z-2}=\cdots=\varepsilon_{z-k+1}=-1$,
		\item $\varepsilon_{z-k}\ne1$.
	\end{enumerate}
\end{lemma}
\begin{proof}
If $(\varepsilon_{k-1},\varepsilon_{k},\dots,\varepsilon_n)\in R_n^{(k)}$ and $\sum_{j=z+1}^n\varepsilon_jF^{(k)}_j=0$, then $\sum_{j=k-1}^z\varepsilon_jF^{(k)}_j=0$. Henceforth we shall refer to this conclusion as the ``modified" zero-condition.\\

Proceeding, given the vector $(\varepsilon_{k-1},\varepsilon_{k},\dots,\varepsilon_n)$, we first isolate the (sub)-vector \newline $(\varepsilon_{k-1},\varepsilon_{k},\dots,\varepsilon_z)$. We further split this new vector into two parts: $(\varepsilon_{k-1},\dots,\varepsilon_v)$ and $(\varepsilon_{v+1},\dots,\varepsilon_z)$. As mentioned in the remark before the statement of the lemma, we will rigorously make use of the intuition that, if the corresponding second part  
$$\sum_{j=v+1}^z\varepsilon_jF^{(k)}_j=V$$ 
is a large enough positive integer, then it would not be possible to ``compensate" $V$ in order to get zero as total sum, even if we set all the coefficients $\varepsilon_i=-1$ 
in the preceding first part. In the proof, this is encapsulated by the strict inequality

$$
\sum_{j=k-1}^vF^{(k)}_j<\,V=\sum_{j=v+1}^z\varepsilon_jF^{(k)}_j.
$$ 

We can now verify the two parts of the theorem.\\

\noindent \textbf{Part (1)}.
\vspace{2mm}

Let us first show that $\varepsilon_{z-1}=-1$. Suppose to the contrary that $\varepsilon_{z-1}=0$ or $1$. This immediately leads to a contradiction of the modified zero-condition since

$$
\sum_{j=k-1}^{z-2}\varepsilon_jF^{(k)}_j\le \sum_{j=k-1}^{z-2}F^{(k)}_j<F^{(k)}_z\le\varepsilon_{z-1}F^{(k)}_{z-1}+F^{(k)}_z,
$$ 
where the second inequality is the direct application of Lemma \ref{l2}.

Now suppose general, that $\varepsilon_{z}=1$ and $\varepsilon_{z-1}=\varepsilon_{z-2}=\cdots=\varepsilon_{z-t}=-1$ for some $t\in\{1,2,\dots,k-2\}$. We would like to show that $\varepsilon_{z-t-1}=-1$ also holds.

To this end we suppose $\varepsilon_{z-t-1}\in\{0,1\}$. It would suffice to show that

\begin{equation}\label{lhs}
	\sum_{j=k-1}^{z-t-2}\varepsilon_jF^{(k)}_j\le \sum_{j=k-1}^{z-t-2}F^{(k)}_j=\sum_{j=k-1}^{z-k-1}F^{(k)}_j+\sum_{j=z-k}^{z-t-2}F^{(k)}_j
\end{equation}
is smaller than
$$
\sum_{j=z-t}^{z-1}(-1)F^{(k)}_j+F^{(k)}_z\le \varepsilon_{z-t-1}F^{(k)}_{z-t-1}+\sum_{j=z-t}^{z-1}(-1)F^{(k)}_j+F^{(k)}_z.
$$
Note that first summand on the right-hand side of (\ref{lhs}) might be empty.
For the left-hand side of the previous inequality we see
\begin{equation}\label{rhs}
	\sum_{j=z-t}^{z-1}(-1)F^{(k)}_j+F^{(k)}_z=\sum_{j=z-k}^{z-t-2}F^{(k)}_j+F^{(k)}_{z-t-1}.
\end{equation}

Removing the common sum $\sum_{j=z-k}^{z-t-2}F^{(k)}_j$ from both (\ref{lhs}) and (\ref{rhs}), we are left with showing that
$$
\sum_{j=k-1}^{z-k-1}F^{(k)}_j<F^{(k)}_{z-t-1}.
$$ 
However, this is true since, by Lemma \ref{l2} and the condition $t\le k-2$, we have
$$
\sum_{j=k-1}^{z-k-1}F^{(k)}_j<F^{(k)}_{z-k+1}\le F^{(k)}_{z-t-1}
$$
if $\sum_{j=k-1}^{z-k-1}F^{(k)}_j$ a non-empty sum, otherwise trivial.\\

\noindent \textbf{Part (2)}.
\vspace{2mm}

The second part can be proved via the same strategy as before, indeed, we show that, assuming $\varepsilon_{z-k}=1$ leads to a contradiction. More precisely, we show that if

\begin{equation}\label{lhsnew}
	\sum_{j=k-1}^{z-k-1}\varepsilon_jF^{(k)}_j\le \sum_{j=k-1}^{z-k-1}F^{(k)}_j<
	F^{(k)}_{z-k}+(-1)\!\!\!\sum_{j=z-k+1}^{z-1}F^{(k)}_j+F^{(k)}_z,
\end{equation} 
then we get a contradiction. The right hand side of (\ref{lhsnew}) is $2F^{(k)}_{z-k}$, and the above inequality is true since the inequality
$$
\sum_{j=k-1}^{z-k-1}F^{(k)}_j<F^{(k)}_{z-k+1}<2F^{(k)}_{z-k}
$$
holds by applying Lemma \ref{l2}, Lemma \ref{l1} and the condition $z-k\ge k-1$, equivalently $z\ge2k-1$.
\end{proof}

Having completed the proof of Lemma \ref{crucial}, we turn our attention to establishing recurrence relations for the terms of the sequences $(r_n)$ and $(r_n^\star)$.

Towards this, consider the last term in the vanishing sum $0=\sum_{j=k-1}^n\varepsilon_jF^{(k)}_j$. This is either 0 or 1 according to the definition of $R_n^{(k)}$.

If $\varepsilon_n=0$, then this case provides $r_{n-1}$ solutions. On the other hand, the situation arising from
\begin{equation}\label{ep=1}
	\varepsilon_n=1
\end{equation} 
requires a more delicate analysis. Applying Lemma \ref{crucial} with $z=n$, we have $\varepsilon_{n-k}$ is either $-1$ or $0$, with the intermediate terms $\varepsilon_i=-1$ for $i=n-k+1,\dots,n-1$. For ease of illustration, we again encode the coefficients via $\{-1, 0, 1 \}$-vectors.\\

{\bf Case 1:} If $\varepsilon_{n-k}=-1$. The complete sum associated to the vector
$$(\dots,-1\overbrace{-1,-1,\dots,-1}^{k-1},1)$$
can be written in the form
$$
0=\sum_{j=k-1}^{n-k-1}\varepsilon_jF^{(k)}_j+\overbrace{(-1)\sum_{j=n-k}^{n-1}F^{(k)}_j+F^{(k)}_n}^{0}=\sum_{j=k-1}^{n-k-1}\varepsilon_jF^{(k)}_j,
$$
where we applied the defining recurrence rule (\ref{reck}) for the terms of $(F^{(k)}_n)$. Clearly, the remaining zero sum gives $r^\star_{n-k-1}$ solutions.\\

{\bf Case 2:} When $\varepsilon_{n-k}=0$. Using the same recurrence rule, the vector

$$(\dots,0,\overbrace{-1,-1,\dots,-1}^{k-1},1),$$

leads to the decomposition

$$
0=\sum_{j=k-1}^{n-k-1}\varepsilon_jF^{(k)}_j+(-1)\sum_{j=n-k+1}^{n-1}F^{(k)}_j+F^{(k)}_n=\sum_{j=k-1}^{n-k-1}\varepsilon_jF^{(k)}_j+F^{(k)}_{n-k}.
$$
Thus $(\dots,0,\overbrace{-1,-1,\dots,-1}^{k-1},1)$ is {\it formally not but identical in value} to $(\dots,1,\overbrace{0,0,\dots,0}^{k-1},0)$. We can thus consider this case as $\varepsilon_{n-k}=1$ and $\sum_{j=n-k+1}^{n}0\cdot F^{(k)}_j=0$. Applying Lemma \ref{crucial} once more, with $z=n-k$, we observe that the current scenario is equivalent to the one we started from, where $\varepsilon_n=1$ (see (\ref{ep=1})), except that $\varepsilon_{n-k}=1$ this time around. We can thus repeat the same argument we had for the case $\varepsilon_n=1$.\\

If $\varepsilon_{n-2k}=-1$, then we gain $r^\star_{n-2k-1}$ additional solutions, while $\varepsilon_{n-2k}=0$ leads analogously to $\varepsilon_{n-2k}=1$, resulting in yet another bifurcation.

Proceeding, if we let $n=qk+u$, $0\le u\le k-1$, we can iteratively repeat the same argument until the condition $2k-1\le z$ holds, from which we get ({\it identically}) that $\varepsilon_{n-qk}=\varepsilon_u=1$.

In summary, through the above process, we have actually showed that
\begin{equation}\label{star}
	r_n=r_{n-1}+r^\star_{n-k-1}+r^\star_{n-2k-1}+\cdots+r^\star_{u}.
\end{equation} 
Equation (\ref{star}) can be simplified further. Expressing $r_{n-k}$ via the equation, (that is, instead of $r_{n-1}$, since $r^\star_{u}$ in (\ref{star}) depends on the remainder of $n$ modulo $k$), we have that
\begin{equation}\label{star-}
	r_{n-k}=r_{n-k-1}+r^\star_{n-2k-1}+r^\star_{n-3k-1}+\cdots+r^\star_{u}
\end{equation}	

Subtracting (\ref{star-}) from (\ref{star}), and recalling that $r^\star_{j}=2r_j+1$, after some basic rearrangements, the equation yields
\begin{equation}\label{rrrr1}
	r_{n}=r_{n-1}+r_{n-k}+r_{n-k-1}+1.
\end{equation}

We obtain the homogeneous version of (\ref{rrrr1}) by applying (\ref{rrrr1}) to $r_{n-1}$ to obtain the corresponding equation, and then evaluating $r_n-r_{n-1}$. After simplifying, we get the desired recurrence relation having of the form
$$
r_n=2r_{n-1}-r_{n-2}+r_{n-k}-r_{n-k-2}.
$$

We have thus proved the following theorem.

\begin{theorem}\label{final}
	The terms of the sequence $(r_n)$ satisfy the recursive rule
	\begin{equation*}\label{oooh}
		r_n=2r_{n-1}-r_{n-2}+r_{n-k}-r_{n-k-2}.
	\end{equation*}
\end{theorem}

Since $2r_n+1=2(2r_{n-1}+1)-(2r_{n-2}+1)+(2r_{n-k}+1)-(2r_{n-k-2}+1)$, as a consequence we also obtain
$$
r^\star_n=2r^\star_{n-1}-r^\star_{n-2}+r^\star_{n-k}-r^\star_{n-k-2},
$$
a recurrence similar to that of $r_n$. The characteristic polynomial of both sequences $(r_n)$ and $(r^\star_n)$ is
$$
c(x)=x^{k+2}-2x^{k+1}+x^k-x^2+1=(x-1)(x^{k+1}-x^k-x-1).
$$

The explicit formula for $r_n$ and $r^\star_n$ depends on the zeros of $c(x)$ and on the initial values. Of course, one zero is known which is $x_1=1$, the other ones cannot be determined in general.

In the next two subsections, we specialize Theorem \ref{final} to the cases $k=2,3$, and examine its consequences more closely.

\subsection{The Fibonacci sequence}

For $k=2$, let
\begin{equation*}
	R^{(2)\star}_n=\{(\varepsilon_1,\varepsilon_2,\dots,\varepsilon_n)\;|\;\varepsilon_1F_1+\varepsilon_2F_2+\cdots+\varepsilon_nF_n=0, \varepsilon_i\in\{-1,0,1\}\},
\end{equation*}
and
$	R^{(2)}_n\subset (R^{(2)\star}_n\setminus(0,0,\dots,0\
))$ defined accordingly. For $n = 1, \ldots, 4$, we find:
\begin{align*}
	&R^{(2)}_1=\emptyset,\quad R^{(2)}_2=\{(-1,1)\},\quad R^{(2)}_3=\{(-1,1,0),(-1,-1,1)\},\\
	&R^{(2)}_4=\{(-1,1,0,0),(-1,-1,1,0),(-1,0,-1,1),(0,-1,-1,1)\}.
\end{align*}

Hence (recalling that $r^\star_n=|R^{(2)\star}_n|$, $r_n=|R^{(2)}_n|$, and $r_n^\star=2r_n+1$) we see that
\begin{equation}\label{inistar}
r_1=0,\; r_2=1,\; r_3=2,\; r_4=4;\quad \text{and} \quad 	r^\star_1=1,\; r^\star_2=3,\; r^\star_3=5,\; r^\star_4=9.
\end{equation}

Noting that $k=2$ gives $-r_{n-2}+r_{n-k}=0$, Theorem \ref{final} now gives:
$$
r_n=2r_{n-1}-r_{n-4}.
$$

It follows that $r^\star_{n}=2r^\star_{n-1}-r^\star_{n-4}$ as well. However, this is not the minimal recursion for $r^\star_{n}$. The initial values (\ref{inistar}) and the factorization 

$$
x^4-2x^3+1=(x-1)(x^3-x^2-x-1)
$$

of the characteristic polynomial show that
$r^\star_n=r^\star_{n-1}+r^\star_{n-2}+r^\star_{n-3}$.

The previous formula together with the initial values from (\ref{inistar}) lead to the recursion
$$
r^\star_n=T_{n+2}+T_n\quad(n\ge1),
$$
where $T_n$ is the Tribonacci sequence. The authors found this discovery rather fascinating. The solution to a problem based on the Fibonacci sequence is given by the Tribonacci numbers. Remarkably, as we will see in the next section, this relationship goes in the reverse direction as well!\\

We thus end the subsection with the following corollary.

\begin{corollary}\label{Fibo}
	The sequences $(r_n)$ and $(r^\star_n)$ satisfy the relations
	\begin{eqnarray*}
		r_n&=&2r_{n-1}-r_{n-4}, \\
		r^\star_n&=&T_{n+2}+T_n.
	\end{eqnarray*}
\end{corollary}
Remark: Note that sequence $(r_n)$ is a  shifted version of A008937 in \cite{OEIS}.

\subsection{The Tribonacci sequence}

For $k=3$, similar to the previous section, we let
\begin{equation*}
	R^{(3)\star}_n=\{(\varepsilon_2,\varepsilon_3,\dots,\varepsilon_n)\;|\;\varepsilon_2T_2+\varepsilon_3T_3+\cdots+\varepsilon_nT_n=0, \varepsilon_i\in\{-1,0,1\}\},
\end{equation*}
and $R^{(3)}_n\subset (S^\star_n\setminus(0,0,\dots,0\
))$ be defined accordingly. It is easy to check that 

\begin{equation}\label{inistar3}
	r_2=0,\quad r_3=1,\quad r_4=2,\quad r_5=3;\qquad r^\star_2=1,\quad r^\star_3=3,\quad r^\star_4=5,\quad r^\star_5=7.
\end{equation}

By Theorem \ref{final}, we get:

\begin{corollary}\label{Tribo}
	The sequence $(r_n)$ satisfies the relation
	\begin{eqnarray*}
		r_n&=&2r_{n-1}-r_{n-2}+r_{n-3}-r_{n-5}. 
	\end{eqnarray*}
\end{corollary}

Note that the characteristic polynomial of this recurrence can be factorized as
$$
x^5-2x^4+x^3-x^2+1=(x-1)(x^2+1)(x^2-x-1),
$$
which foreshadows the claimed connection with the Fibonacci sequence. Indeed, using the zeros $1,i,-i,\alpha=(1+\sqrt{5})/2,\beta=(1-\sqrt{5})/2$ of the characteristic polynomial and the first four initial values from (\ref{inistar3}), we obtain the explicit formula

$$
r_n=-\frac{1}{2}-\frac{3+i}{20}i^{n-2}+\frac{-3+i}{20}(-i)^{n-2}+\frac{2+\sqrt{5}}{5}\alpha^{n-2}+\frac{2-\sqrt{5}}{5}\beta^{n-2}.
$$

Since $2+\sqrt{5}=\alpha^3$, $2-\sqrt{5}=\beta^3$, and noting that the first three terms in the above sum take the values $\delta_n=-1/5,-3/5,-4/5,-2/5$ if $n$ is congruent to $0,1,2,3$, respectively, the previous formula can be further simplified to get:

$$
r_n=\frac{L_{n+1}}{5}+\delta_n=\frac{F_{n+2}+F_{n}}{5}+\delta_n,
$$

where $(L_n)$ is the sequence of Lucas numbers with explicit formula $L_n=\alpha^n+\beta^n$. We highlight the interesting observation that a relationship uncovered in the previous section now seems to be going the other way: therein we had showed that the number of Fibonacci representations was related to Tribonacci numbers, and have now showed that the number of Tribonacci representations is also linked to Fibonacci numbers.\\

The first few terms of $(r_n)$ (for $n\ge2$) are
$$
0,1,2,3,5,9,15,24,39,64,104,168,272,441,714,1155,\dots
$$	
which is seemingly the sequence A097083 in \cite{OEIS}. On an ``algebraic level", this connection is perhaps not so surprising since the characteristic polynomial $(x-1)(x^2 + 1)(x^2 -x -1)$ of $r_n$ is the denominator of the summands of the generating function of A097083. However, connections on the combinatorial level remain speculative currently.

\vspace{0.3cm}

\section{Random inhomogeneous Tribonacci recurrence}

In this section, we follow the same strategy we previously employed for the Fibonacci sequence case in \cite{LSz}. The proof of Theorem \ref{t1} will be given in detail, while immediate corollaries will just be stated. When necessary, we will also occasionally highlight certain differences and similarities to the paper \cite{LSz}.

\subsection{Entries of row $n$}

This subsection is devoted to investigating the features of the binary tree illustrated by Figure \ref{fig1}. We will start by determining the general entry of the tree $({\mathcal T}_{n,k})$. To prevent potential clash of notation, when the binary representation of a number is being invoked, it will be indicated/denoted by an upper index $ ^{[2]}$.

\begin{theorem}\label{t1} (analogous to Theorem 1 of \cite{LSz})
	Let $n\ge3$ and $0\le k\le 2^{n-2}-1$. Assume that the binary representation of $k$ is $k=\varepsilon_{n-3}\varepsilon_{n-4}\dots\varepsilon_{1}\varepsilon_{0}^{[2]}
	$ ($\varepsilon_i\in\{0,1\}$). The entry ${\mathcal T}_{n,k}$ of the $k$th element of row $n$ is given by
	\begin{equation}\label{e1}
		{\mathcal T}_{n,k}=m_a(n)+\left(b-a\right)\sum_{j=0}^{n-3}\varepsilon_jT_{j+2}.
	\end{equation}	
\end{theorem}

\begin{proof}

	First let $n=3$. Then for $k=0=0^{[2]}$ we have ${\mathcal T}_{3,0}=T_3+a(T_4+T_2-1)/2+(b-a)\cdot0\cdot T_2=1+a$, while $k=1=1^{[2]}$ yields ${\mathcal T}_{3,1}=T_3+a(T_4+T_2-1)/2+(b-a)\cdot1\cdot T_2=1+b$. Now let $n=4$. We have four cases $k=0,1,2,3$. For example, when $k=2=10^{[2]}$, we have that
	
	$$
	{\mathcal T}_{4,2}=T_4+a\frac{T_5+T_3-1}{2}+(b-a)(0\cdot T_2+1\cdot T_3)=2+a+b.
	$$
	
The other three entries of row 3 are obtained similarly: we have ${\mathcal T}_{4,0}=2+2a$, ${\mathcal T}_{4,1}=2+a+b$, and ${\mathcal T}_{4,3}=2+2b$. One can similarly check the values ${\mathcal T}_{5,0}=4+4a$, ${\mathcal T}_{5,1}={\mathcal T}_{5,2}=4+3a+b$, ${\mathcal T}_{5,3}={\mathcal T}_{5,4}=4+2a+2b$, ${\mathcal T}_{5,5}={\mathcal T}_{5,6}=4+a+3b$, and ${\mathcal T}_{5,7}=4+4b$ (see  Figure \ref{fig1}).
	
Assume that the statement is true for $3,4,\dots,n-1$ $(n\ge6)$. Put $k_1=\lfloor k/2\rfloor$, $k_2=\lfloor k_1/2\rfloor$, and $k_3=\lfloor k_2/2\rfloor$. The formula
	
\begin{equation*}\label{derivation}
		{\mathcal T}_{n,k}={\mathcal T}_{n-1,k_1}+{\mathcal T}_{n-2,k_2}+{\mathcal T}_{n-3,k_3}+\left((b-a)\varepsilon_0+a\right)	
\end{equation*}
	
gives the derivation rule of the tree. (Clearly, $(b-a)\varepsilon_0+a$ is equal to $a$ if  $\varepsilon_0=0$, and $b$ otherwise.) Applying this, together with the induction hypothesis and the definition of the Tribonacci sequence yields

	\begin{eqnarray*}
		{\mathcal T}_{n,k}&=&T_{n-1}+a\frac{T_{n}+T_{n-2}-1}{2}+(b-a)\sum_{j=1}^{n-3}\varepsilon_jT_{j+1}\\
		&&+T_{n-2}+a\frac{T_{n-1}+T_{n-3}-1}{2}+(b-a)\sum_{j=2}^{n-3}\varepsilon_jT_{j}\\
		&&+T_{n-3}+a\frac{T_{n-2}+T_{n-4}-1}{2}+(b-a)\sum_{j=3}^{n-3}\varepsilon_jT_{j-1}+(b-a)\varepsilon_0+a\\
		&=&T_{n}+a\frac{T_{n+1}+T_{n-1}-1}{2}-a+(b-a)\sum_{j=3}^{n-3}\varepsilon_jT_{j+2}\\ &&+(b-a)(\varepsilon_1T_2+\varepsilon_2T_3)+(b-a)\varepsilon_2T_2+((b-a)\varepsilon_0+a)\\
		&=&T_{n}+a\frac{T_{n+1}+T_{n-1}-1}{2}+(b-a)\sum_{j=0}^{n-3}\varepsilon_jT_{j+2}.
	\end{eqnarray*}

One can easily check that $(b-a)(\varepsilon_1T_2+\varepsilon_2T_3+\varepsilon_2T_2+\varepsilon_0=(b-a)(\varepsilon_0T_2+\varepsilon_1T_3+\varepsilon_2T_4)$ holds above, and this completes the proof.
\end{proof}	

Below we give two corollaries that parallel analogous results previously established for the Fibonacci case in \cite{LSz}. The same arguments therein essentially translate to our current case.

\begin{corollary}\label{cn1} (analogous to Corollary 1 in \cite{LSz}, and Corollary 1 of  \cite{LSz0})

Let $a_0=a$, $a_1=b$. Then the formula (\ref{e1}) is equivalent to
	$${\mathcal T}_{n,k}=T_n+\sum_{j=0}^{n-3}a_{\varepsilon_j}T_{j+2}.$$
\end{corollary}

\begin{corollary}\label{cn2} (analogous to Corollary 2 in \cite{LSz})	

	Taking the $n^{th}$ row of the tree (for $n \geq 3$), as $k$ goes through the range $0,1,\dots,2^{n-2}-1$, the formula (\ref{e1}) gives a map onto the set  
	
	$$\left\{m_a(n),\,m_a(n)+(b-a),\,\dots,\,m_a(n)+j(b-a),\,\dots,\, m_b(n)\right\}.$$
	
The cardinality of the number of distinct elements in the $n^{th}$ row of the tree is 

	$$T_n^\star=\frac{T_{n+1}+T_{n-1}+1}{2}.$$
\end{corollary}

The next three statements describe the relation between the first and second half of a row, the first part of a row and the previous row, and finally the second half of a row and the previous row, respectively (for illustration see Figure \ref{fig8}).
\vspace{0.3cm}

\begin{figure}[h]
	\centering
	\includegraphics[scale=0.5]{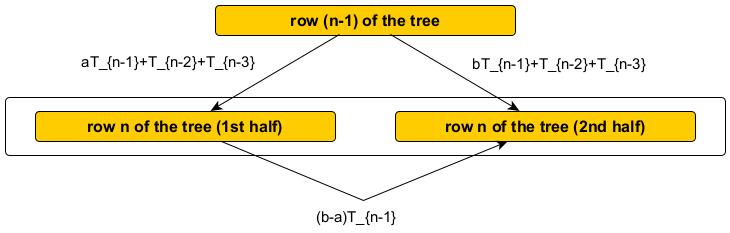} 
	\caption{Relation among row $n-1$ and half parts of row $n$.}
	\label{fig8}
\end{figure}

\begin{corollary}\label{cn3} (analogous to Corollary 3 in \cite{LSz})
	If $n\ge 3$ and $0\le k\le 2^{n-3}-1$, $K=k+2^{n-3}$, then
	$$
	{\mathcal T}_{n,K}={\mathcal T}_{n,k}+(b-a)T_{n-1}.
	$$	
\end{corollary}	

\begin{theorem}\label{tn2} (analogous to Theorem 2 in \cite{LSz}, and partially to  Corollary 1 of \cite{LSz0})
	If $n\ge 3$ and $0\le k\le 2^{n-3}-1$, then 
	\begin{equation*}
		{\mathcal T}_{n,k}={\mathcal T}_{n-1,k}	+(aT_{n-1}+T_{n-2}+T_{n-3}).
	\end{equation*}
\end{theorem}

\begin{corollary}\label{cn4} (analogous to Corollary 4 in \cite{LSz}, and partially to  Corollary 1 of \cite{LSz0})
	If $n\ge 3$ and $0\le k\le 2^{n-3}-1$, $K=k+2^{n-3}$, then
	$$
	{\mathcal T}_{n,K}={\mathcal T}_{n-1,k}+(bT_{n-1}+T_{n-2}+T_{n-3}).
	$$	
\end{corollary}

\subsection{Graph transformation of $({\mathcal T}_{n,k})$}
The purpose of this subsection is to study a new graph $({\mathcal F})$ which can be obtained from the tree by fusing the vertices having identical value in a row. These fusions inevitably result in vertices with multiplicities, which we interpret as frequencies. The method is verbatim that which was detailed in \cite{LSz}. The reader interested in the exact techniques is encouraged to consult the paper. Meanwhile, we exhibit the graph transformation via the below figure.

\begin{figure}[h]
	\centering
	\includegraphics[scale=0.4]{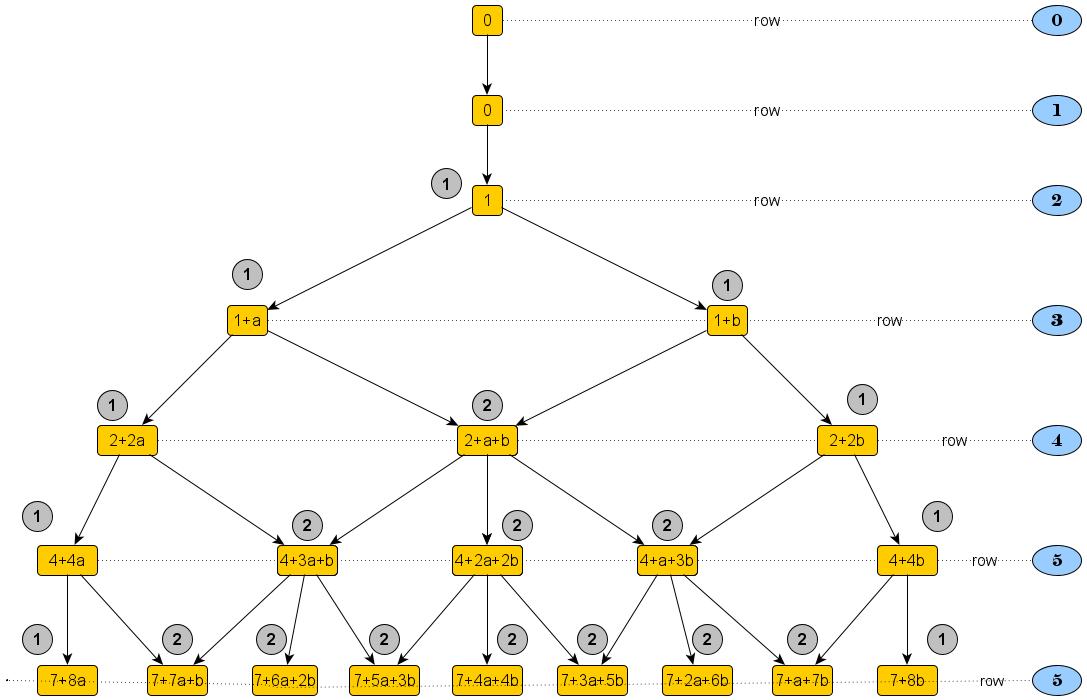} 
	\caption{The first few rows of graph $({\mathcal F})$.}
	\label{fig2}
\end{figure}

After the process of graph transformation the vertices of the new graph have frequencies. Figure \ref{fig2} shows the frequencies of the vertices in gray circles. For example, the vertex $7+7a+b$ in row 5 has frequency 2, which means that the value $7+7a+b$ appears 2 times in row 5 of the tree (cf.~Figure \ref{fig1}).

Assume that the table
\begin{equation*}\label{gameta}
	\Gamma_n=\; \left\{ \begin{array}{cccc}
		v_0 & v_1 & \dots & v_{T^\star_{n}-1} \\
		\gamma_0 & \gamma_1 & \dots & \gamma_{T^\star_{n}-1} \\
	\end{array} \right\}
\end{equation*}
gives the values $v_i$ of row $n$ of $({\mathcal F})$ with frequencies $\gamma_i$, respectively. Here $v_i=T_n+a(T^\star_n-1)+(b-a)i$, $i=0,1,\dots,T^\star_n-1$.
(See Corollary \ref{cn2}.) The polynomial $$Q_n(x)=\gamma_0+\gamma_1x+\cdots+\gamma_{T^\star_{n}-1}x^{T^\star_{n}-1}$$ helps keep track of the coefficients. For instance, $Q_2(x)=1$, $Q_3(x)=1+x$, $Q_4(x)=1+2x+x^2$. We can obtain the following theorem.

\begin{theorem}\label{tn333} (analogous to Theorem 3 in \cite{LSz})
	The polynomials  $Q_{n-1}(x)$ and $Q_n(x)$ satisfy the equality
	\begin{equation*}
		Q_{n+1}(x)=Q_{n}(x)(1+x^{T_{n}}).
	\end{equation*}
\end{theorem}

A direct consequence of the previous theorem is
\begin{theorem}\label{tn444} (analogous to Theorem 4 in \cite{LSz})
	\begin{equation*}\label{qqriq}
		Q_n(x)=\prod_{k=2}^{n-1}(1+x^{T_k}).
	\end{equation*}
\end{theorem}

Remark: This formula shows that the distribution of representation counts is governed by the Tribonacci numbers themselves, producing a self-similar coefficient structure visible in Figure \ref{fig3}. For illustration, we show the coefficients of $Q_{11}(x)$ in a bar diagram. The fractal structure is evident.

\begin{figure}[h]
	\centering
	\includegraphics[scale=0.4]{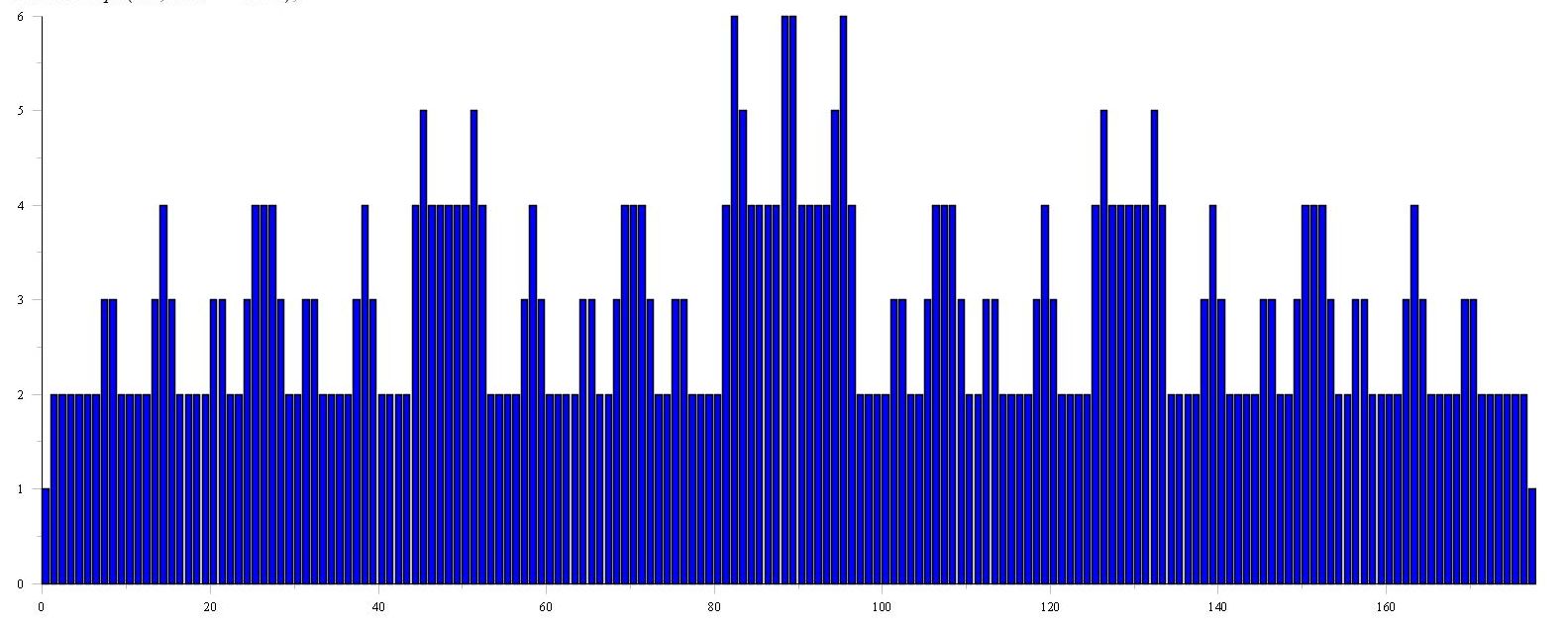} 
	\caption{Bar diagram of the coefficients of $Q_{11}(x)$.}
	\label{fig3}
\end{figure}

\section{A probabilistic interpretation and first-order statistics}

\noindent We now interpret the polynomial $Q_n(x) = \prod_{k=2}^{n-1} (1 + x^{T_k})$ as the probability generating function of a random variable.\\

\noindent Let $(\varepsilon_k)_{k \geq 2}$ be a sequence of independent Bernoulli random variables with
\begin{equation*}
\mathbb{P}(\varepsilon_k = 1) = \mathbb{P}(\varepsilon_k = 0) = \tfrac{1}{2}.
\end{equation*}

\noindent Define the random sum $S_n := \sum_{k=2}^{n-1} \varepsilon_k T_k$.

\begin{prop}[Probabilistic representation]
The coefficients of $Q_n(x)$ define the probability mass function of $S_n$. More specifically, 

\begin{equation*}
\frac{Q_n(x)}{2^{n-2}} = \mathbb{E}[x^{S_n}] = \sum_{m \geq 0} \mathbb{P}(S_n = m)x^m.
\end{equation*}
\end{prop}

\begin{proof}

By independence, the generating function of $\varepsilon_k T_k$ is $\mathbb{E}[x^{\varepsilon_k T_k}] = \tfrac{1}{2}(1 + x^{T_k})$.\\

Taking products over $k$ yields $\mathbb{E}[x^{S_n}] = \prod_{k=2}^{n-1} \tfrac{1}{2}(1 + x^{T_k})$, which agrees with $Q_n(x)$ up to normalization. Since $Q_n(1) = 2^{n-2}$, dividing by this factor yields the stated probability generating function.
\end{proof}

\noindent This representation shows that $S_n$ is a sum of independent random variables supported on $\{0, T_k\}$, and hence its distribution is a finite convolution of Bernoulli measures.

\medskip

\subsection*{Mean and variance}

We now compute the first two moments of $S_n$.

\begin{prop}
The expectation and variance of $S_n$ are given by $\mathbb{E}[S_n] = \frac{1}{2} \sum_{k=2}^{n-1} T_k$ and $\mathrm{Var}(S_n) = \frac{1}{4} \sum_{k=2}^{n-1} T_k^2$.

\end{prop}

\begin{proof}
By linearity of expectation and independence, $\mathbb{E}[S_n] = \sum_{k=2}^{n-1} \mathbb{E}[\varepsilon_k T_k]
= \sum_{k=2}^{n-1} \tfrac{1}{2} T_k$.\\

Similarly, $\mathrm{Var}(S_n) = \sum_{k=2}^{n-1} \mathrm{Var}(\varepsilon_k T_k)
= \sum_{k=2}^{n-1} T_k^2 \cdot \mathrm{Var}(\varepsilon_k)
= \frac{1}{4} \sum_{k=2}^{n-1} T_k^2$.

\end{proof}

\subsection*{A heuristic observation}

\noindent
To study asymptotic behaviour, we normalize by the $T_n$ and define $X_n := \frac{S_n}{T_n}$. Using the asymptotic growth $T_k \sim C \rho^k$, where $\rho > 1$ is the Tribonacci constant, we obtain $\frac{T_k}{T_n} \sim \rho^{k-n}$. Thus, $X_n \approx \sum_{k=2}^{n-1} \varepsilon_k \rho^{k-n}$.\\

\noindent Reindexing via $j = n-k$, we obtain the representation $X_n \approx \sum_{j=1}^{n-2} \varepsilon_{n-j} \rho^{-j}$.\\

\noindent This observation suggests convergence to an infinite random series, which we make precise in the next subsection.

\subsection*{Limiting distribution}

We now make the heuristic argument above precise.

\begin{theorem}[Limiting distribution]

Let $X_n := \frac{S_n}{T_n}$. Then $X_n$ converges in distribution to the random variable

\begin{equation*}
X := \sum_{j=1}^{\infty} \varepsilon_j \rho^{-j},
\end{equation*}

where $(\varepsilon_j)_{j \geq 1}$ are independent Bernoulli$(\tfrac{1}{2})$ random variables.
\end{theorem}

\begin{proof}

We write

\begin{equation*}
X_n = \sum_{k=2}^{n-1}\varepsilon_k\frac{T_k}{T_n}.
\end{equation*}

Reindexing with $j=n-k$, we obtain

\begin{equation*}
X_n = \sum_{j=1}^{n-2} \varepsilon_{n-j}\frac{T_{n-j}}{T_n}.
\end{equation*}

Since the random variables $\varepsilon_k$ are i.i.d., for every $n$ we have

\begin{equation*}
X_n \overset{d}{=} Y_n := \sum_{j=1}^{n-2}
\varepsilon_j\frac{T_{n-j}}{T_n}.
\end{equation*}

We shall show that

\begin{equation*}
Y_n\longrightarrow
X := \sum_{j=1}^{\infty}\varepsilon_j\rho^{-j}
\end{equation*}

almost surely. This will imply the desired convergence in distribution.\\

Fix $J\geq 1$. From the asymptotic relation

\begin{equation*}
T_m\sim C\rho^m
\qquad (m\to\infty),
\end{equation*}

it follows, for every fixed $j$, that

\begin{equation*}
\frac{T_{n-j}}{T_n}
\longrightarrow
\rho^{-j}
\qquad (n\to\infty).
\end{equation*}

Consequently,

\begin{equation*}
\sum_{j=1}^{J} \varepsilon_j\frac{T_{n-j}}{T_n}
\longrightarrow \sum_{j=1}^{J}\varepsilon_j\rho^{-j}
\end{equation*}

almost surely as $n\to\infty$.\\

It remains to control the tails uniformly. Since
$T_m\sim C\rho^m$, there exist constants $c,C_1>0$
and $m_0\geq 1$ such that

\begin{equation*}
c\rho^m\leq T_m\leq C_1\rho^m
\qquad (m\geq m_0).
\end{equation*}

After rescaling the constant to absorb the finitely many terms
with index less than $m_0$, there exists $C_2>0$ such that,
for all sufficiently large $n$,

\begin{equation*}
\begin{aligned}
\sum_{j=J+1}^{n-2}\frac{T_{n-j}}{T_n}
&=
\frac{1}{T_n}\sum_{k=2}^{n-J-1}T_k\\
&\leq
C_2\rho^{-J}.
\end{aligned}
\end{equation*}

Since $0\leq\varepsilon_j\leq 1$, it follows that

\begin{equation*}
\left|
\sum_{j=J+1}^{n-2}
\varepsilon_j\frac{T_{n-j}}{T_n}
\right|
\leq C_2\rho^{-J}.
\end{equation*}

On the other hand,

\begin{equation*}
\left|
\sum_{j=J+1}^{\infty}\varepsilon_j\rho^{-j}
\right| \leq \sum_{j=J+1}^{\infty}\rho^{-j}
= \frac{\rho^{-J}}{\rho-1}.
\end{equation*}

Therefore,

\begin{equation*}
\begin{aligned}
|Y_n-X|
\leq {}&
\left|
\sum_{j=1}^{J}
\varepsilon_j
\left(
\frac{T_{n-j}}{T_n}-\rho^{-j}
\right)
\right|\\
&\quad
+C_2\rho^{-J}
+\frac{\rho^{-J}}{\rho-1}.
\end{aligned}
\end{equation*}

For fixed $J$, the first term tends to zero almost surely as
$n\to\infty$. Hence

\begin{equation*}
\limsup_{n\to\infty}|Y_n-X|
\leq C_2\rho^{-J}
+\frac{\rho^{-J}}{\rho-1}
\qquad\text{almost surely}.
\end{equation*}

Letting $J\to\infty$ gives

\begin{equation*}
Y_n\longrightarrow X
\qquad\text{almost surely}.
\end{equation*}

Since $X_n\overset{d}{=}Y_n$ for every $n$, we conclude that
\begin{equation*}
X_n\xrightarrow{d}X.
\end{equation*}

Finally, the limiting random series is well defined, since
\begin{equation*}
0\leq
\sum_{j=1}^{\infty}\varepsilon_j\rho^{-j}
\leq
\sum_{j=1}^{\infty}\rho^{-j}
=
\frac{1}{\rho-1}<\infty.
\end{equation*}
Thus the series defining $X$ converges absolutely (indeed, for
every realization of the Bernoulli variables).
\end{proof}
\medskip

\noindent
The limiting random variable $X$ is a classical example of a Bernoulli convolution with parameter $\rho^{-1}$.

\medskip

\subsection{\bf Self-similarity and distributional equation}

The limiting distribution satisfies a natural self-similarity relation.

\begin{prop}[Distributional fixed-point equation]
Let $X$ be as above. Then
\begin{equation*}
X \stackrel{d}{=} \rho^{-1}( X + \varepsilon),
\end{equation*}
where $\varepsilon \sim \mathrm{Bernoulli}(\tfrac{1}{2})$ is independent of $X$.
\end{prop}

\begin{proof}

We write
\begin{equation*}
X = \sum_{j=1}^{\infty} \varepsilon_j \rho^{-j}
= \rho^{-1} \varepsilon_1 + \sum_{j=2}^{\infty} \varepsilon_j \rho^{-j}.
\end{equation*}

Factoring out $\rho^{-1}$ from the second sum gives

\begin{equation*}
X = \rho^{-1} \left( \varepsilon_1 + \sum_{j=2}^{\infty} \varepsilon_j \rho^{-(j-1)} \right).
\end{equation*}

Reindexing the sum,
\begin{equation*}
\sum_{j=2}^{\infty} \varepsilon_j \rho^{-(j-1)} \stackrel{d}{=} X.
\end{equation*}
Thus,
\begin{equation*}
X \stackrel{d}{=} \rho^{-1} X + \rho^{-1} \varepsilon_1  =\rho^{-1}(X+\varepsilon_1).
\end{equation*}

Since $\varepsilon_1\sim\operatorname{Bernoulli}(1/2)$ and is independent of $\sum_{j=2}^{\infty}\varepsilon_j\rho^{-(j-1)}$, the result follows upon writing $\varepsilon=\varepsilon_1$.
\end{proof}

\subsubsection*{Remark}
This identity shows that the law of $X$ is invariant under a simple affine random transformation. In particular, it can be viewed as the invariant measure of an iterated function system with maps $x \mapsto \rho^{-1} x$ and $x \mapsto \rho^{-1} x + \rho^{-1}$.\\

This perspective explains the fractal features observed in the distribution of $S_n/T_n$, and connects the present setting to the extensive literature on Bernoulli convolutions and self-similar measures.

\vspace{0.3cm}

\vspace{0.3cm}

\section{Acknowledgments}

For L.~Szalay the research was supported by Hungarian National Foundation for Scientific Research Grant No.~150284, the Slovak Scientific Grant Agency VEGA 1/0493/25, and and the National Research, Development and Innovation Office Grant 2023-1.2.4-T\'ET-2023-00063. He would like to express his greatest gratitude for the hospitality during his staying at the University of Pretoria. Special thanks to Maya Thackeray. T.P Chalebgwa's research was supported by the University of Pretoria's Research Development Program (RDP) grant.

\end{document}